\documentclass[a4paper]{amsart}
\usepackage[latin1]{inputenc}
\usepackage{amssymb}
\usepackage{amsmath}
\usepackage{mathrsfs}
\usepackage{eufrak}
\usepackage{amsthm}
\usepackage{amsfonts}
\usepackage{textcomp}
\usepackage{graphicx}
\usepackage[pdftex]{color}
\usepackage{paralist}
\usepackage[shortlabels]{enumitem}
\usepackage{hyperref}
\usepackage{comment}
\usepackage[arrow, matrix, curve]{xy}
\usepackage{tikz}
\usepackage{tikz-cd}

\newtheorem*{corollary*}{Corollary}
\newtheorem*{conjecture*}{Conjecture}
\newtheorem*{example*}{Example}
\newtheorem*{theorem*}{Theorem}
\newtheorem*{proposition*}{Proposition}

\newtheorem{theorem}{Theorem}[section]

\newtheorem{lemma}[theorem]{Lemma}

\newtheorem*{claim*}{Claim}
\newtheorem*{conjecture}{Conjecture}
\newtheorem*{theorem2}{Theorem}

\theoremstyle{definition}

\theoremstyle{remark}

\numberwithin{equation}{section}

\makeatletter
\renewcommand*\env@matrix[1][\
arraystretch]{%
  \edef\arraystretch{#1}%
  \hskip -\arraycolsep
  \let\@ifnextchar\new@ifnextchar
  \array{*\c@MaxMatrixCols c}}
\makeatother

\begin{document}

\title{On total stability conditions for Dynkin quivers}
\date{\today}



\author{Ren\'{e} Marczinzik}
\address{Mathematical Institute of the University of Bonn, Endenicher Allee 60, 53115 Bonn, Germany}
\email{marczire@math.uni-bonn.de}

\begin{abstract}
We show that for a Dynkin quiver $Q$ of type $E_7$ with a specific orientation, the path algebra $KQ$ has no slope function of the form $\mu=\frac{\theta}{\dim}$ that defines a total stability condition. This gives a counterexample to a conjecture of Reineke.
\end{abstract}

\maketitle
\section*{Introduction}
Let $A$ be a finite dimensional algebra over a field $K$ with Grothendieck group $K_0(A)$.
A \emph{slope function for $A$}, $\mu: K_0(A) \setminus \{0\} \rightarrow \mathbb{R}$ is defined as the quotient $\mu=\frac{\theta}{\kappa}$, where $\theta: K_0(A) \rightarrow \mathbb{R}$ and $\kappa: K_0(A) \rightarrow \mathbb{R}$ are linear functions on $K_0(A)$ such that $\kappa(M)>0$ for all $0 \neq M \in K_0(A)$.
One of the most important classes of slope functions are those of the form $\mu=\frac{\theta}{\dim}$, where $\dim$ associates to $M$ its vector space dimension over $K$.
A non-zero $A$-module $M$ is called \emph{$\mu$-stable} if $\mu(N)<\mu(M)$ for all non-zero proper submodules $N$ of $M$.
Slope functions are closely related to stability condition and stability structures on the module category of $A$ in the sense of King and Rudakov, see \cite{K} and \cite{Rud}. In recent years stability conditions on abelian and triangulated categories have appeared in many contexts and have interactions with fields such as algebraic geometry, topology and mathematical physics, we refer for example to \cite{B} and \cite{Q} for more on stability conditions and related topics.
A slope function is said to define a \emph{total stability condition} if every indecomposable $A$-module is $\mu$-stable.

One of the most important classes of finite dimensional algebras are path algebras $KQ$ for a finite quiver $Q$. The fundamental theorem of Gabriel states that $KQ$ is of finite representation type if and only if $Q$ is of Dynkin type.
In \cite{Rei}, Reineke posed the following as conjecture 7.1: 
\begin{conjecture}
Let $A=KQ$ be a path algebra of Dynkin type.
Then there exists a slope function that defines a total stability condition of the form $\mu=\frac{\theta}{\dim}$ for $A$.
\end{conjecture}

I made a program in the GAP-package \cite{QPA} that translates the problem of finding a slope function that defines a total stability condition for a given Dynkin quiver $Q$ into an elementary problem about inequalities. It turned out that the conjecture is true for all $Q$ of Dynkin type having at most six points. But surprisingly a counterexample was found for a specific $Q$ of Dynkin type $E_7$ with the following orientation:
$Q=$
\[\begin{tikzcd}
	&& 7 \\
	1 & 2 & 3 & 4 & 5 & 6
	\arrow["{a_1}", from=2-2, to=2-1]
	\arrow["{a_2}", from=2-3, to=2-2]
	\arrow["{a_3}", from=2-4, to=2-3]
	\arrow["{a_4}", from=2-5, to=2-4]
	\arrow["{a_5}", from=2-6, to=2-5]
	\arrow["{a_6}", from=2-3, to=1-3]
\end{tikzcd}\]
The main result of this article is the following counterexample to the conjecture of Reineke.
\begin{theorem2}
Let $Q$ be the quiver of Dynkin type $E_7$ as above and $A=KQ$ the path algebra of $Q$. 
Then $A$ has no slope function of the form $\mu=\frac{\theta}{\dim}$ that defines a total stability condition.

\end{theorem2}

We present an elementary proof that can be verified by hand.
We only assume that the reader is familiar with Auslander-Reiten theory and the classification of indecomposable modules for path algebra $KQ$ of Dynkin type, where we refer for example to the textbook \cite{ARS} for an elementary introduction.

\section{A counterexample to the conjecture of Reineke}
Let $Q$ be the quiver of Dynkin type $E_7$ with the orientation as above and $A=KQ$ its path algebra for the rest of this section. We use right modules unless stated otherwise.
We identity indecomposable modules with their dimension vectors since each indecomposable module is uniquely determined by its dimension vector for path algebras of Dynkin type.
We leave the proof of the following lemma to the reader as it is an easy consequence of the known classification of indecomposable modules for $KQ$ and the shape of the Auslander-Reiten quiver of $KQ$. For the convinience of the reader we show how to verify the lemma using the GAP-package \cite{QPA} in the next section as an appendix.
\begin{lemma} \label{lemma}
There exist the following five inclusions between indecomposable modules in $A$:
\begin{enumerate}
\item $f_1: [0,0,1,1,1,0,1] \rightarrow [1,1,2,2,2,1,1]$
\item $f_2: [0,1,1,1,0,0,0] \rightarrow [0,1,2,2,1,0,1]$
\item $f_3: [0,0,1,0,0,0,0] \rightarrow [0,0,1,1,0,0,0]$
\item $f_4: [0,1,1,0,0,0,0] \rightarrow [0,1,2,1,0,0,1]$
\item $f_5: [1,1,1,1,1,1,1] \rightarrow [1,2,2,1,1,1,1]$

\end{enumerate}
In all five cases, the inclusion is an irreducible morphism.
\end{lemma}

We can now prove our main result:
\begin{theorem}
Let $Q$ be the quiver of Dynkin type $E_7$ as above and $A=KQ$ the path algebra of $Q$. 
Then $A$ has no slope function of the form $\mu=\frac{\theta}{\dim}$ that defines a total stability condition.
\end{theorem}
\begin{proof}
Assume there exists a slope function of the form $\mu=\frac{\theta}{\dim}$ that defines a total stability condition. Since $\theta$ is linear we can write $\theta$ applied to a dimension vector $[y_1,y_2,y_3,y_4,y_5,y_6,y_7]$ as follows for some $x_1,...,x_7 \in \mathbb{R}$:
$\theta([y_1,y_2,y_3,y_4,y_5,y_6,y_7])=x_1y_1+x_2y_2+x_3y_3+x_4y_4+x_5y_5
+x_6y_6+x_7y_7$, for all $y_1,...,y_7 \in \mathbb{R}$.

Then by the previous lemma, we have the following inequalities since $\mu$ is assumed to be total:
\begin{enumerate}
\item $\mu([0,0,1,1,1,0,1])<\mu([1,1,2,2,2,1,1])$
\item $\mu([0,1,1,1,0,0,0])<\mu([0,1,2,2,1,0,1])$
\item $\mu([0,0,1,0,0,0,0])<\mu([0,0,1,1,0,0,0])$
\item $\mu([0,1,1,0,0,0,0])<\mu([0,1,2,1,0,0,1])$
\item $\mu([1,1,1,1,1,1,1])<\mu([1,2,2,1,1,1,1])$.

\end{enumerate} 

Using the definition of $\mu$, we see that the five inequalities are equivalent to the following five inequalities:
\begin{enumerate}
\item $0<4x_1+4x_2-2x_3-2x_4-2x_5+4x_6-6x_7$
\item $0<-4x_2-x_3-x_4+3x_5+3x_7$
\item $0<-x_3+x_4$
\item $0<-3x_2-x_3+2x_4+2x_7$
\item $0 < -2x_1 +5x_2 + 5x_3 - 2x_4 -2x_5-2x_6-2x_7$.

\end{enumerate}
We will show now that this system of five inequalties has no solution in the real numbers.
Multiply the inequality (1) by $\frac{1}{2}$ and add the result to the inequality (5) to obtain the following inequalities:
\begin{enumerate}
\item $0<4x_1+4x_2-2x_3-2x_4-2x_5+4x_6-6x_7$
\item $0<-4x_2-x_3-x_4+3x_5+3x_7$
\item $0<-x_3+x_4$
\item $0<-3x_2-x_3+2x_4+2x_7$
\item[(5')] $0 <  7x_2 + 4x_3 - 3x_4 -3x_5-5x_7$.

\end{enumerate}
Now add the inequalities (2) and (4) to (5') so that we obtain the system of five inequalities as follows:
\begin{enumerate}
\item $0<4x_1+4x_2-2x_3-2x_4-2x_5+4x_6-6x_7$
\item $0<-4x_2-x_3-x_4+3x_5+3x_7$
\item $0<-x_3+x_4$
\item $0<-3x_2-x_3+2x_4+2x_7$
\item[(5'')] $0<2x_3-2x_4$

\end{enumerate}

Now clearly there is no solution to (3) and (5'') combined so the original five inequalities also have no solution in the real numbers.
Thus we see that no total function of the form $\mu=\frac{\theta}{\dim}$ can exist that defines a total stability condition.
\end{proof}

\section{Appendix: Proof of the lemma using QPA}
In this appendix we show how to prove lemma \ref{lemma} using the GAP-package \cite{QPA}.
If you have a recent version of GAP installed, QPA is automatically available and can be loaded inside GAP via the command
\begin{verbatim}
LoadPackage("qpa");
\end{verbatim}
We will use the following QPA-program (copy and paste it into QPA) that gives all indecomposable modules for a Dynkin type path algebra $A=KQ$: \newline
\begin{verbatim}
DeclareOperation("IndModDynkin",[IsList]);

InstallMethod(IndModDynkin, "for a representation of a quiver", [IsList],0,function(LIST)

local A,C,n,injA,W,i,WW,l;

A:=LIST[1];
C:=CoxeterMatrix(A);
n:=Order(C);
injA:=IndecInjectiveModules(A);
W:=[];for i in injA do for l in [0..n] do Append(W,[DTr(i,l)]);od;od;
WW:=Filtered(W,x->Dimension(x)>0);
return(WW);

end);
\end{verbatim}
The program uses the fact that for $A=KQ$ of Dynkin type, all indecomposable modules are of the form $\tau^l(I)$ for an indecomposable injective $A$-module $I$ and some $l$ with $0 \leq l \leq n$, where $n$ is the order of the Coxeter matrix of $A$. Some of the modules of the form $\tau^l(I)$ might be zero but the output only gives the non-zero indecomposable modules.
The next command defines the path algebra $A=KQ$ of Dynkin type $E_7$ as in the main text and calculates all indecomposable modules for $A$ (we work over the field $K$ with three elements here, but the representation theory of $A$ and its Auslander-Reiten quiver do not depend on the field):
\begin{verbatim}
K:=GF(3);
Q:=Quiver(["1","2","3","4","5","6","7"], [["2","1","a_1"],["3","2","a_2"] 
,["4","3","a_3"],["5","4","a_4"],["6","5","a_5"],["3","7","a_6"]]);
A:=PathAlgebra(K,Q);L:=IndModDynkin([A]);Size(L);
\end{verbatim}
After entering this into QPA, we see that the list $L$ of indecomposable modules contains 63 indecomposable modules, which is also what the theory predicts.

The indecomposable $A$-module $M_1$ with dimension vector $[1,1,2,2,2,1,1]$ and its almost split sequence  $0 \rightarrow \tau(M_1) \rightarrow X_1 \rightarrow M_1 \rightarrow 0$ can now be obtained as follows:
\begin{verbatim}
M:=L[11];W1:=AlmostSplitSequence(M);g1:=W1[2];X1:=Source(g1);DecomposeModule(X1);
\end{verbatim}
Here $g1$ is the map $X_1 \rightarrow M_1$ in the almost split sequence and the last command decomposes the module $X_1$ into indecomposable modules.
The result is that $X_1$ is the direct sum of the indecomposable modules with dimension vectors  $[ 0, 0, 1, 1, 1, 0, 1 ]$ and $[ 1, 2, 3, 3, 2, 1, 1 ]$.
Thus we see that there is an irreducible map from $[ 0, 0, 1, 1, 1, 0, 1 ]$ to $[1,1,2,2,2,1,1]$ and since every irreducible map is either injective or surjective, this must be an injective map. Thus we verified (1) of lemma \ref{lemma}.
The verification of parts (2) to (5) of lemma \ref{lemma} can be done in exactly the same way and we leave this to the reader.
We remark that the indecomposable module with dimension vector [0,1,2,2,1,0,1] is L[12] in the above list $L$, the indecomposable module with dimension vector [0,0,1,1,0,0,0] is L[39], the indecomposable module with dimension vector [0,1,2,1,0,0,1] is L[43] and the indecomposable module with dimension vector [1,2,2,1,1,1,1] is L[44].
\section*{Acknowledgements} 
Rene Marczinzik gratefully acknowledges funding by the DFG (with project number 428999796). This project profited from the use of the GAP-package \cite{QPA} and \cite{Sage}. The author is thankful to Max Alekseyev for helpful comments related to Sage.

\end{document}